\documentclass[11pt,reqno,twoside]{amsart}
\usepackage{amssymb,amsmath,amsthm,soul,color,paralist}
\usepackage{t1enc}
\usepackage[cp1250]{inputenc}
\usepackage{a4,indentfirst,latexsym}
\usepackage{graphics}

\usepackage{mathrsfs}
\usepackage{cite,enumerate,graphicx}
\usepackage[english]{babel}
\usepackage[metapost]{mfpic}
\usepackage[colorinlistoftodos]{todonotes}
\usepackage[normalem]{ulem}

\usepackage{etoolbox}
\usepackage[english]{babel} 

\makeatletter
\patchcmd{\@settitle}{\uppercasenonmath\@title}{}{}{}
\patchcmd{\@setauthors}{\MakeUppercase}{}{}{}
\patchcmd{\section}{\scshape}{}{}{}
\makeatother

\usepackage{titlesec}
\titleformat{\section}
  {\large\bfseries}{\thesection}{1em}{}
\usepackage[left=2.45cm,right=2.45cm,top=2.75cm,bottom=2.75cm]{geometry}

\usepackage[colorlinks=true,urlcolor=blue,
citecolor=red,linkcolor=blue,linktocpage,pdfpagelabels,
bookmarksnumbered,bookmarksopen]{hyperref}

\numberwithin{equation}{section}

\newcommand{\ri}{\rightarrow}
\newcommand{\R}{\mathbb{R}}

\newcommand{\J}{\mathcal{J}}
\newcommand{\gr}{\nabla}

\newcommand{\wri}{\rightharpoonup}

\newcommand{\Lp}{L^{p}(\R^{N})}

\newcommand{\Lam}{-\Delta_{p}}
\newcommand{\Y}{W^{1,p}(\R^{N})}
\newcommand{\X}{W_{0}^{1,p}(\Omega)}
\newcommand{\Jl}{\mathcal{J}_{\lambda}}
\newcommand{\Il}{\mathcal{I}_{\lambda}}
\newcommand{\la}{\lambda}
\newcommand{\I}{\mathcal{I}}
\newcommand{\Yr}{W^{1,p}_{\text{rad}}(\R^{N})}
\newcommand{\Xr}{W_{0,\text{rad}}^{1,p}(\Omega)}
\newcommand{\rad}{\text{rad}}

\DeclareMathOperator{\supp}{supp}

\newtheorem{prop}{Proposition}[section]
\newtheorem{lem}{Lemma}[section]
\newtheorem{thm}{Theorem}[section]

\makeatletter
\newcommand{\leqnomode}{\tagsleft@true}
\newcommand{\reqnomode}{\tagsleft@false}
\makeatother

\date{}

\begin{document}

\title[]{\Large Normalized solution for $p$-Laplacian equation in exterior domain}
\author{{Weiqiang Zhang}$^{a,b}$\quad{Yanyun Wen}$^{c*}$
 \bigskip \\
$^{a}$School of Mathematics and Statistics, Ningxia University, Yinchuan 750021, China\\
$^{b}$Ningxia Basic Science Research Center of Mathematics, Ningxia University, Yinchuan 750021, China\\
$^{c}$School of Department of Mathematics,
Gansu Normal College For Nationalities, Hezuo 747000, China\\
}

\thanks{
E-mail: zhangwq@nxu.edu.cn (W. Zhang);  wenyy19@lzu.edu.cn (Y. Wen).
\newline \indent *Corresponding author.
}
\maketitle
\begin{abstract}
We are devoted to the study of the following nonlinear $p$-Laplacian Schr\"odinger equation with $L^{p}$-norm constraint 
\begin{align*}
\begin{cases}
&-\Delta_{p} u=\lambda |u|^{p-2}u +|u|^{r-2}u\quad\mbox{in}\quad\Omega,\\
&u=0\quad\mbox{on}\quad \partial\Omega,\\
&\int_{\Omega}|u|^{p}dx=a,
\end{cases}
\end{align*}
where $\Delta_{p}u=\text{div} (|\nabla u|^{p-2}\nabla u)$, $\Omega\subset\mathbb{R}^{N}$ is an exterior domain with smooth boundary $\partial\Omega\neq\emptyset$ satisfying that $\R^{N}\setminus\Omega$ is bounded, $N\geq3$, $2\leq p<N$, $p<r<p+\frac{p^{2}}{N}$, $a>0$ and $\lambda\in\R$ is an unknown Lagrange multiplier. First, by using the splitting techniques and the Gagliardo-Nirenberg inequality, the compactness of Palais-Smale sequence of the above problem at higher energy level is established. Then, exploiting barycentric function methods, Brouwer degree and minimax principle, we obtain a solution $(u,\la)$ with $u>0$ in $\R^{N}$ and $\la<0$ when $\R^{N}\setminus\Omega$ is contained in a small ball. Moreover, we give a similar result if we remove the restriction on $\Omega$ and     
assume $a>0$ small enough. Last, with the symmetric assumption on $\Omega$, we use genus theory to consider infinite many solutions.

\noindent
\textbf{Key Words}: $p$-Laplacian, normalized solution, exterior domain.

\noindent
\textbf{Mathematics Subject Classification}: 35B09, 35J62, 35J92.
\end{abstract}

\section{Introduction}
In this paper we investigate the following nonlinear $p$-Laplacian Schr\"odinger equation with $L^{p}$-norm constraint 
\begin{align}\label{p}
\begin{cases}
&-\Delta_{p} u=\lambda |u|^{p-2}u +|u|^{r-2}u\quad\mbox{in}\quad\Omega,\\
&u=0\quad\mbox{on}\quad \partial\Omega,\\
&\int_{\Omega}|u|^{p}dx=a,
\end{cases}
\end{align}
where $\Delta_{p}u=\text{div} (|\nabla u|^{p-2}\nabla u)$, $\Omega\subset\mathbb{R}^{N}$ is an exterior domain with smooth boundary $\partial\Omega\neq\emptyset$ satisfying that $\R^{N}\setminus\Omega$ is bounded, $N\geq3$, $2\leq p<N$, $p<r<p_{c}$ with $p_{c}=p+\frac{p^{2}}{N}$, $a>0$ and $\lambda\in\R$ is an unknown Lagrange multiplier. 

The research motivation of \eqref{p} arises from seeking the stationary solutions for reaction diffusion problems which appears in biophysics, plasma physics, and chemical reaction design. The prototype equation for these
models can be written in the form
$$
u_{t}=\Lam u+\la|u|^{p-2}u+|u|^{r-2}u,
$$
where $u$ generally stands for a concentration, the term $\Lam u$ corresponds to the diffusion with coefficient and $\la|u|^{p-2}u+|u|^{r-2}u$ represents the reaction term related to source and loss processes, see \cite{PRS2023}.

When $\Omega=\R^{N}$, we can rewrite equation \eqref{p} as follows 
\begin{align}\label{pR}
\begin{cases}
&-\Delta_{p} u=\lambda |u|^{p-2}u +|u|^{r-2}u\quad\mbox{in}\quad\R^{N},\\
&\int_{\R^{N}}|u|^{p}dx=a.
\end{cases}
\end{align}
There is a amount of research of solutions for this problem. It is well know that under the $L^{p}$-norm constraint, by the Gagliardo-Nirenberg inequality, there appears a $L^{p}$-critical exponent $p_{c}$. For  $L^{p}$-subcritical growth nonlinearity, the problem of seeking solution of problem \eqref{pR} can be reformulated as finding the global minimum of a variational functional on a constraint. However, when
the nonlinearity is $L^{p}$-supcritical growth, the variational functional is unbounded below  under the constraint. Thus one can not obtain a solution by the methods of finding a global minimum and the method one usually uses is to establish the mountain pass geometry under the restriction.
Now, we briefly present the results of study for equation \eqref{pR}.  By Schwarz rearrangement and
Ekeland variational principle,  Zhang-Zhang \cite{ZZ2022Non} obtained the existence of positive radial ground states solution with $L^{p}$-subcritical growth and $L^{p}$-supcritical growth terms, they also dealt with  the case where $L^{p}$-subcritical growth, $L^{p}$-critical growth and $L^{p}$-supcritical growth terms co-existed,
and they used the fountain theorem to obtain the existence of infinitely many nonradial sign-changing solutions. Inspired by Bartsch-Molle-Rizzi \cite{BMR2021}, in which the authors studied  the normalized solution for  Schr\"odinger equation with the affect of potential,
Deng and Wu \cite{DW2023p}  investigated the existence of normalized solution in the presence of potentials
and pointed out that the solution they obtained is  mountain pass type. Moreover, they in \cite{DW2023} also 
considered the existence of solution and infinitely many solutions for \eqref{pR} with Sobolev exponent and $L^{p}$-subcritical exponent. For more researches on \eqref{pR}, we refer the readers to \cite{CR2024,LZ2024,SW2024}.

When $p=2$, $\la<0$ is a fixed constant and without the $L^{p}$-norm constraint to \eqref{p}, this equation becomes the Schr\"odinger equation
\begin{align}\label{sch}
-\Delta u=\la u+|u|^{r-2}u\quad\mbox{in}\quad\Omega.
\end{align}
The study of  existence of nontrivial solutions for \eqref{sch} has a number of literature in the exterior domain $\Omega$, we cite to \cite{ABT2020,Ad2017,AAT2021}. As we know, there few work to \eqref{sch} under the  constant 
\begin{align}\label{cos}
\int_{\Omega}|u|^{2}dx=a.
\end{align}
Precisely, Zhang-Zhang \cite{ZZ2022} investigated equation \eqref{sch}-\eqref{cos} if the nonlinearity is $L^{2}$-subcritical growth. They provided a method to show the compactness of Palas-Smale sequence for the corresponding variational functional. 
Later, Yu-Tang-Zhang \cite{YTZ2023,YTZ2023Z} extended and improved the work of Zhang-Zhang \cite{ZZ2022} to fractional Laplacian equation and Choquard equation.  We note that Song-Hajaiej \cite{SH2024} considered \eqref{sch}-\eqref{cos} in the exterior of a ball
with $L^{2}$-supercritical nonlinearity, which is a first literature to deal with this problem under $L^{2}$-supercritical case.
In \cite{LM2023}, Lancelotti-Molle derived the existence of normalized solution with potential in $L^{2}$-supercritical growth.

In this paper, we shall use the variational approaches to study the existence of solutions for \eqref{p}. The variational functional of \eqref{p} is defined by 
$$
\J(u)=\frac{1}{p}\int_{\Omega}|\gr u|^{p}dx-\frac{1}{r}\int_{\Omega}|\gr u|^{r}dx.
$$
We consider this problem on the manifold
$$
S(a)=\{u\in W^{1,p}(\Omega):\int_{\Omega}|u|^{p}dx=a\}.
$$
In order to study the solution of \eqref{p}, a natural idea is to look at the existence of global minimum  for the following constraint problem 
\begin{align}\label{mi}
m(a)=\inf\{\J(u): u\in S(a)\}.
\end{align}
The following result shows that this global minimum does not exist.

\begin{thm}\label{thm1}
The minimization problem \eqref{mi} is not attained.
\end{thm}

When $\Omega$ is replaced by $\R^{N}$, we know that $m(a)$ is attained, see \cite{CJS2010,G2018}. The fact that \eqref{mi} is not attained is a crucial feature of elliptic 
equation on exterior domain, which implies that the ground state of \eqref{p} does not exist and we can not use the strategy in  
\cite{CJS2010,G2018} to investigated the existence of nontrivial solution.

To explore the existence of solution of \eqref{p}, we borrow the ideals in \cite{ZZ2022,YTZ2023,YTZ2023Z}. 
First, we establish the compactness of Palas-Smale sequence. The classical Lions concentration compactness \cite{L1984} no longer applies since $\Omega$ is a domain. Besides, the unboundedness of the domain $\Omega$ implies that the working space $\X$ is embedded in $L^{s}(\R^{N})$  for any $s\in(p,p^{*})$ non compactly. Thus we have to  overcome the difficulties posed by these two issues.
We shall use the splitting lemma to get the compactness of Palas-Smale sequence of $\J$ at the levels closed to $m(a)$. In this process,  the uniqueness of solution for the problem
$$
\Lam u=\la u+|u|^{r-2}u\quad\mbox{in}\quad\R^{N}
$$
plays an important role. Later, we exploit the barycenter function techniques to derive a Palas-Smale sequence of $\J$ at the levels closed to $m(a)$. Differently from \cite{ZZ2022,YTZ2023,YTZ2023Z}, since $p$-Laplacian is nonlinear, we need a more careful analysis 
to finish the compactness. Especially, we prove the  almost everywhere convergence for the gradients of Palas-Smale sequence.

We state the rest of results as follows.

\begin{thm}\label{thm2}
Suppose that $a>0$, $N\geq 3$, $2\geq p<N$ and $p<r<p_{c}$. If there exists $\rho=\rho(a)>0$ being a small constant depending on $a$, such that if
 $\Omega^{c}\subset B_{\rho}(0)$, then problem \eqref{p} has a positive solution $u\in S(a)$ with $\lambda<0$.
\end{thm}

\begin{thm}\label{thm3}
Suppose that $N\geq 3$, $2\geq p<N$, $p<r<p_{c}$ and $\Omega \subset\R^{N}$ is an exterior domain. Then there exists a constant $a_{*}>0$ such that for any $0<a\leq a^{*}$, problem \eqref{p} has a positive solution $u\in S(a)$ with $\lambda<0$.
\end{thm}

When $\Omega=\R^{N}\setminus\overline{B_{1}(0)}$, according to the symmetric minimax theorem  developed by Jeanjean-Lu \cite{JL2019},
we  can show the existence of infinitely many radial solutions with negative energy.

\begin{thm}\label{thm4}
Suppose that $a>0$, $N\geq 3$, $2\geq p<N$, $p<r<p_{c}$ and $\Omega=\R^{N}\setminus\overline{B_{1}(0)}$. Then problem \eqref{p} has a positive radial solution $u\in S(a)$ with $\lambda<0$. Moreover, there exists infinitely many solutions $\{(u_{n},\lambda_{n})\}$ such that $u_{n}$ is radial, $\lambda_{n}<0$, 
$\J(u_{n})<0$ and 
$$\J(u_{n})\ri0^{-}\quad\mbox{as}\quad n\ri\infty.
$$
\end{thm}

The organization of this paper is as follows. In section \ref{sec2}, we give some preliminary results and the non-existence of ground state. In section \ref{sec3}, we present the compactness of Palas-Smale sequence. In section \ref{sec4}, we complete the proof of Theorem \ref{thm2} and Theorem \ref{thm3}. In section \ref{sec5}, we prove Theorem \ref{thm4}.

Thought out this paper, let $o_{n}(1)$  represent $o_{n}(1)\ri 0$ as $n\ri\infty$, and $\rightharpoonup$ and $\ri$ denote the weak convergence and the strong convergence in the corresponding spaces respectively.

\section{Preliminaries}\label{sec2}

Let $p\in[1,\infty]$ and $\Lambda\subset\R^{N}$ be a measurable set. We use $L^{p}(\Lambda)$ to denote the usual Lebesgue space that $u:\Lambda \ri\R$ is a measurable function and $\|u\|_{L^{p}(\Lambda)}<\infty$,
where 
\begin{align*}
\|u\|_{L^{p}(\Lambda)}=
\begin{cases}
(\int_{\Lambda}|u|^{p}dx)^{\frac{1}{p}}&\quad\mbox{for}\quad p\in[1,\infty),\\
\text{esssup}_{x\in\Lambda}|u(x)|&\quad\mbox{for}\quad p=\infty.
\end{cases}
\end{align*}
For $p\in(1,N)$, let $\mathcal{D}^{1,p}(\Lambda)$ denote the closure of $C_{c}^{\infty}(\R^{N})$
with respect to the norm
$$
\|u\|_{L^{p}(\Lambda)}=\bigg(\int_{\Lambda}|u|^{p}dx\bigg)^{\frac{1}{p}},
$$
which is equivalent to 
$$
\mathcal{D}^{1,p}(\Lambda)=\bigg\{u:\Lambda\ri\R\quad\mbox{is measurable}:\|u\|_{L^{s}(\Lambda)}<\infty\bigg\}.
$$ 
We define the Sobolev space  $W^{1,p}(\Lambda)=L^{p}(\Lambda)\cap\mathcal{D}^{1,p}(\Lambda)$ endowed with the 
norm
$$
\|u\|_{W^{1,p}(\Lambda)}=(\|u\|_{L^{p}(\Lambda)}^{p}+\|\gr u\|_{L^{p}(\Lambda)}^{p})^{\frac{1}{p}}.
$$

We know that $W^{1,p}(\Lambda)$ is a  reflexive and separable Banach space, $W^{1,p}(\Lambda)$
is continuously embedded in the Lebesgue space $L^{t}(\Lambda)$ for any $t\in[p,p^{*}]$ and compactly embedded
in $L_{\text{loc}}^{t}(\Lambda)$ for any $t\in[p,p^{*})$, see \cite{A1975}.

For brevity, let $W^{1,p}(\Omega)$ equip with the norm
$$
\|u\|=(\|u\|_{p}^{p}+\|\gr u\|_{p}^{p})^{\frac{1}{p}}.
$$

Recall some basic properties for the following Schr\"odinger equation
\begin{equation}\label{pi}
\Lam u+ |u|^{p-2}u=|u|^{r-2}u\quad\mbox{in}\quad\R^{N}.
\end{equation}

By virtue of \cite{HL2014,ST2000,DW2023}, one can derive the following properties.
\begin{lem}\label{un}
Assume that $p<r<p^{*}$. Then up to translations, problem \eqref{pi} has a unique positive ground state solution $w\in C^{2}(\R,\R)$, which is radial, decreasing and  satisfies 
$$
w(x)\leq Ce^{-\kappa |x|},
$$
for some $\kappa>0$.
\end{lem}

We consider the limit problem of \eqref{p}
\begin{align}\label{li}
\begin{cases}
&-\Delta_{p} u=\lambda |u|^{p-2}u +|u|^{r-2}u\quad\mbox{in}\quad\R^{N},\\
&\int_{\R^{N}}|u|^{p}dx=a.
\end{cases}
\end{align}

It is well-known that the solution of \eqref{li} is the variational functional
$\I:\Y\ri\R$ restricted on the manifold 
$$
S_{*}(a):=\{u\in\Y:\int_{\R^{N}}|u|^{p}dx=a\},
$$
where 
$$
\I(u)=\frac{1}{p}\int_{\R^{N}}|\gr u|^{p}dx-\frac{1}{r}\int_{\R^{N}}|u|^{r}dx.
$$

Let $\lambda<0$. We define the following two functionals
$$
\Jl(u)=\frac{1}{p}\int_{\Omega}|\gr u|^{p}dx
+\frac{-\lambda}{p}\int_{\Omega}|u|^{p}dx-\frac{1}{r}\int_{\Omega}|u|^{r}dx
$$
and 
$$
\Il(u)=\frac{1}{p}\int_{\R^{N}}|\gr u|^{p}dx
+\frac{-\lambda}{p}\int_{\R^{N}}|u|^{p}dx-\frac{1}{r}\int_{\R^{N}}|u|^{r}dx.
$$

\begin{prop}\label{ma}
For any $a>0$, $m_{*}(a):=\inf\{\I(s):u\in S(a)\}$ is attained by a unique (up to a translation) positive radial minimizer $w_{\la}$ satisfying 
$$
w_{\la}(x)=(-\la)^{\frac{1}{r-p}}w((-\la)^{\frac{1}{p}}x),
$$
where $w_{\la}$ is the unique positive solution of 
$$
\Lam u=\la|u|^{p-2}u+|u|^{r-2}u\quad\mbox{in}\quad\R^{N},
$$
and $\la<0$ is determined by
\begin{align}\label{ma1}
\|w_{\la}\|^{p}_{L^{p}(\R^{N})}=(-\la)^{s_{p,r}}\|w\|_{L^{p}(\R^{N})}^{p},\quad s_{p,r}=\frac{p}{r-p}-\frac{N}{p}.
\end{align}
Moreover, there holds 
\begin{align}\label{ma2}
m_{*}(a)=(-\la)^{s_{p,r}+1}\I(w)=(\frac{a}{\|w\|_{L^{p}(\R^{N})}})^{\frac{s_{p,r}+1}{s_{p,r}}}\I(w).
\end{align}
\end{prop}
\begin{proof}
From \cite{ZZ2022Non}, we know that $m_{*}(a)$ is attained for any $a>0$. Then there exists $(w_{\la},\la)$ being a solution of \eqref{li} and $\la<0$. Note that by the change of variable, the solutions of \eqref{pi} are in one-to-one correspondence with the solutions of \eqref{li}. Thus we can directly utilize Lemma \ref{un} to complete the proof. 
\end{proof}

Recall the following splitting lemma which is firstly proposed by  Struwe \cite{S1984}.
\begin{lem}\label{spl}
Let $\lambda<0$. If $\{u_{n}\}\subset \X$ is a $(PS)_{c}$ sequence of $\Jl$, then in the sense of a subsequence, 
there exist a nonnegative integer $k$, $u_{0}\in\X$, 
and $k$ sequences $\{y_{n}^{i}\}$ such that $|y_{n}^{i}|\ri\infty$ as $n\ri\infty$, $1\leq i\leq k$, such that $u_{0}$ is a weak solution of equation
\begin{align*}
&\Lam u=\la|u|^{p-2}u+|u|^{r-2}u\quad\mbox{in}\quad\Omega,\\
&u_{n}\wri u\quad\mbox{in}\quad\X,\\
&u_{n}=u_{0}+\sum_{i=1}^{k}w_{\la}(x+y_{n}^{i})\quad\mbox{strongly in}\quad\Y,\\
&\|u_{n}\|^{p}\ri \|u_{0}\|^{p}+k\|w_{\la}\|_{\Y}^{p},\\
&\Jl(u_{n})\ri\Jl(u_{0})+k\Il(w_{\la}).
\end{align*}
\end{lem}

Motivated by \cite{ZZ2022}, we introduce the map $\Psi_{a,R}:\R^{N}\ri S_{*}(a)$ defined as
$$
\Psi_{a,R}(y)=\frac{\sqrt[p]{a}}{\|\eta(\frac{\cdot}{R})w_{\la}(\cdot-y)\|_{L^{p}(\R^{N})}}
\eta(\frac{\cdot}{R})w_{\la}(\cdot-y)\quad\mbox{for}
\quad y\in\R^{N},
$$  
where $a,\,R>0$, $p<r<p_{c}$, $\la=\la(a)$ is determined by Proposition \ref{ma}, $w_{\la}$ is a positive minimizer of \eqref{li} and $\eta:\R^{N}\ri[0,1]$ is a smooth radial function such that
$$
\eta\equiv0\quad\mbox{for}\quad |x|\leq1,\quad\eta(x)\equiv1\quad\mbox{for}\quad|x|\geq2.
$$
We present some properties of $\Psi_{a,R}$
as follows.

\begin{lem}\label{be}
There holds that

$(i)$ $\Psi_{a,R}\in C(\R^{N},S_{*}(a))$ for $R>0$;

$(ii)$ $\I(\Psi_{a,R}(y))\ri m_{*}(a)$ as $|y|\ri\infty$ uniformly for $R>0$ is bounded;

$(iii)$ $\Psi_{a,R}\ri w_{\la}$ in $\Y$ as $R\ri0^{+}$ uniformly for $y\in\R^{N}$.
\end{lem}
\begin{proof}
$(i)$ It is easy to conclude this assertion. We omit it here.

$(ii)$ By the change of variable $z=x-y$, we derive 
$$
\int_{\R^{N}}|\eta(\frac{x}{R})w_{\la}(x-y)|^{p}dx
=\int_{\R^{N}}|\eta(\frac{z+y}{R})w_{\la}(z)|^{p}dz
$$
From the Lebesgue dominated convergence theorem, we deduce that
\begin{align}\label{beii1}
\int_{\R^{N}}|\eta(\frac{z+y}{R})w_{\la}(z)|^{p}dz\ri
\int_{\R^{N}}|w_{\la}(z)|^{p}dz=a\quad\mbox{as}\quad |y|\ri\infty.
\end{align}
Similarly, by the change of variable $z=x-y$, one has that 
\begin{align}\label{beii2}
\int_{\R^{N}}|\eta(\frac{x}{R})w_{\la}(x-y)|^{r}dx
=\int_{\R^{N}}|\eta(\frac{z+y}{R})w_{\la}(z)|^{r}dz\ri
\int_{\R^{N}}|w_{\la}(z)|^{r}dz\quad\mbox{as}\quad|y|\ri\infty.
\end{align} 
Besides, we can see that
\begin{align}\label{beii3}
&\quad\int_{\R^{N}}|\gr(\eta(\frac{x}{R})w_{\la}(x-y)-w_{\la}(x-y))|^{p}dx\nonumber\\
&=\int_{\R^{N}}|\gr[(\eta(\frac{z+y}{R})-1)w_{\la}(z)]|^{p}dz\nonumber\\
&=\int_{\R^{N}}|\frac{1}{R}\gr\eta(\frac{z+y}{R})w_{\la}(z)+(\eta(\frac{z+y}{R})-1)\gr w_{\la}(z)|^{p}dz\nonumber\\
&\leq\frac{2^{p-2}}{R^{p}}\int_{\R^{N}}|\gr\eta(\frac{z+y}{R})|^{p}|w_{\la}(z)|^{p}dz
+2^{p-1}\int_{\R^{N}}|\eta(\frac{z+y}{R})-1|^{p}|\gr w_{\la}(z)|^{p}dz
\ri0\quad\mbox{as}\quad|y|\ri\infty.
\end{align}
In the light of \eqref{beii1}-\eqref{beii3}, we can obtain that for each $R>0$,
$$
\I(\Psi_{a,R}(y))\ri m_{*}(a)\quad\mbox{as}\quad |y|\ri\infty.
$$ 

$(iii)$ 
By the fact that $w_{\la}$ is radial and decreasing, we have that for each $a>0$, 
\begin{align}\label{beiii1}
&\quad\int_{\R^{N}}|\eta(\frac{x}{R})w_{\la}(x-y)-w_{\la}(x-y)|^{p}dx\nonumber\\
&=\int_{\R^{N}}|(\eta(\frac{z+y}{R})-1)w_{\la}(z)|^{p}dx\nonumber\\
&\leq\int_{B_{2R}(-y)}|w_{\la}|^{p}dz\leq\int_{B_{2R}(0)}|w_{\la}|^{p}dz\ri0\quad\mbox{as}\quad R\ri0.
\end{align}
On the other hand, one can deduce from the H\"older inequality that
\begin{align}\label{beiii2}
&\quad\int_{\R^{N}}|\gr[\eta(\frac{x}{R})w_{\la}(x-y)-w_{\la}(x-y)]|^{p}dx\nonumber\\
&=\int_{\R^{N}}|\frac{1}{R}\gr\eta(\frac{x}{R})w_{\la}(x-y)+(\eta(\frac{x}{R})-1)\gr w_{\la}(x-y)|^{p}dx\nonumber\\
&=\int_{\R^{N}}|\frac{1}{R}\gr\eta(\frac{z+y}{R})w_{\la}(z)+(\eta(\frac{z+y}{R})-1)\gr w_{\la}(z)|^{p}dx\nonumber\\
&\leq\frac{2^{p-1}}{R^{p}}\int_{B_{2R}(-y)}|w_{\la}|^{p}dz+
2^{p-1}\int_{B_{2R}(-y)}|\gr w_{\la}|^{p}dz\nonumber\\
&\leq\frac{2^{p-1}}{R^{p}}\bigg(\int_{B_{2R}(-y)}|w_{\la}|^{p^{*}}dz\bigg)^{\frac{p}{p^{*}}}
\bigg(\int_{B_{R}(y)}dz\bigg)^{\frac{p}{N}}+
2^{p-1}\int_{B_{2R}(-y)}|\gr w_{\la}|^{p}dz\nonumber\\
&\leq C(N)\bigg(\int_{B_{2R}(-y)}|w_{\la}|^{p^{*}}dz\bigg)^{\frac{p}{p^{*}}}+
2^{p-1}\int_{B_{2R}(-y)}|\gr w_{\la}|^{p}dz\nonumber\\
&\ri0\quad\mbox{as}\,\, R\ri0
\quad\mbox{uniformly for}\quad y\in\R^{N}.
\end{align}
\eqref{beiii1} and \eqref{beiii2} guarantee that
$$
\Psi_{a,R}\ri w_{\la}\,\,\mbox{ in}\,\, \Y\,\,\mbox{ as}\,\, R\ri0^{+}\,\,\mbox{ uniformly for}\,\, y\in\R^{N}.
$$
\end{proof}

\begin{lem}\label{min}
We have that $m(a)=m_{*}(a)<0$.
\end{lem}
\begin{proof}
Exploiting the fact $S(a)\subset S_{*}(a)$, we obtain that $m_{*}(a)\leq m(a)$. Next we show that $m(a)\leq m_{*}(a)$. Noticing that $\R^{N}\backslash\Omega$ is bounded, then there exists $R>0$ such that
$\R^{N}\backslash\Omega\subset B_{R}(0)$. Since $\eta(\frac{x}{R})\equiv0$ for $|x|\leq R$, by the meaning of $\Psi_{a,R}(y)$ we can see that $\supp(\Psi_{a,R}(y))\subset B_{R}^{c}(0)\subset\Omega$ for each $y\in\R^{N}$, which combined with $\Psi_{a,R}(y)\in S_{*}(a)$ implies that $\Psi_{a,R}(y)\in S(a)$ for any $y\in\R^{N}$.
Using $\Psi_{a,R}(y)\subset S(a)$ for any $y\in\R^{N}$ and Lemma \ref{be}-$(ii)$, we conclude that
$$
m(a)\leq\lim_{|y|\ri\infty}\J(\Psi_{a,R}(y))=\lim_{|y|\ri\infty}\I(\Psi_{a,R}(y))=m_{*}(a).
$$
\end{proof}

\begin{proof}[{\bf Proof of Theorem \ref{thm1}}]
On contrary we assume that there exists $u\in S(a)$ such that $\J(u)=m(a)$. We may assume that $u\geq0$ in $\Omega$ due to $|u|\in S(a)$ and $\J(|u|)=m(a)$. It follows from Lemma \ref{min} that 
$$
\I(u)=\J(u)=m(a)=m_{*}(a).
$$
Thus $u$ is a minimizer of $\I$ on $S_{*}(a)$. Then there exists $\la<0$ such that $u$ is a weak solution of the equation
$$
\Lam u=\la |u|^{p-2}u+|u|^{r-2}\quad\mbox{in}\quad\R^{N}.
$$
By the maximum principle and the fact $u\geq0$ in $\Y$ we can conclude that $u>0$ in $\R^{N}$, which is a contradiction due to $u\equiv0$ in $\R^{N}\setminus\Omega$. The proof is finished.
\end{proof}

\section{Compactness results}\label{sec3}
In this section, we are devoted to characterize the compactness of Palais-Smale sequence at some levels close to $m(a)$.

Recall that the Gagliardo-Nirenberg inequality: for any $r\in(p,p^{*})$, there exists a sharp constant $C_{N,p,r}$ such that
\begin{align}\label{gn}
\|u\|_{L^{r}(\R^{N})}\leq C_{N,p,r}\|\gr u\|^{\gamma_{p,r}}_{L^{p}(\R^{N})}\|u\|^{1-\gamma_{p,r}}_{L^{p}(\R^{N})}\quad\mbox{for any}\quad u\in\Y,
\end{align}
where
$$
\gamma_{p,r}:=N(\frac{1}{p}-\frac{1}{r})=\frac{N(r-p)}{rp}.
$$

\begin{lem}\label{co}
Suppose that $a>0$, $p<r<p_{c}$ and $\Omega\subset\R^{N}$ is an exterior domain. Then there exists a positive constant $\xi=\xi(a)\in(p^{-1/s_{p}},1)$ depending on a such that if $\{u_{n}\}\subset S(a)$ is a nonnegative
$(PS)_{c}$ sequence of $\J|_{S(a)}$ at the level $c$ with $m_{*}(a)<c<\xi m_{*}(a)$, then up to  a subsequence there exists $u\in S(a)$ such that
$$
u_{n}\ri u_{0}\quad\mbox{in}\quad\X.
$$ 
Besides, $u_{0}$ is a positive solution of \eqref{p} with $\la<0$.
\end{lem}
\begin{proof}
Let $\{u_{n}\}$ be a nonnegative Palais-Smale sequence of $\J$ restricted on $S(a)$ at the level $c$ with
$$
c\in(m_{*}(a),p^{-1/s_{p}}m_{*}(a)).
$$ 
This leads to that
\begin{align}\label{co1}
c=\frac{1}{p}\|\gr u_{n}\|_{p}^{p}-\frac{1}{r}\|u_{n}\|_{r}^{r}+o_{n}(1)
\end{align}
and there exists $\la_{n}\in\R$ such that for any $\varphi\in\X$,
\begin{align}\label{psca}
\int_{\Omega}|\gr u_{n}|^{p-2}\gr u_{n}\gr\varphi dx
-\int_{\Omega}|u_{n}|^{r-2}u_{n}\varphi dx
=\la_{n}\int_{\R^{N}}|u_{n}|^{p-2}u_{n}\varphi dx+o_{n}(1)\|\varphi\|.
\end{align}
By Gagliardo-Nirenberg inequality \eqref{gn} and \eqref{co1}, we can infer that $\{u_{n}\}$ is bounded in $\X$.
Thus there holds
\begin{align}\label{co2}
\|\gr u_{n}\|_{p}^{p}
-\|u_{n}\|_{r}^{r}
=\la_{n}a+o_{n}(1).
\end{align}
Using the boundedness of $\{u_{n}\}$ in $\X$, we deduce that $\{\la_{n}\}$ is bounded. We assume that $\la_{n}\ri\la$. Combining \eqref{co1} and \eqref{co2}, one has 
\begin{align*}
\la_{n}a+o_{n}(1)=\|\gr u_{n}\|_{p}^{p}+cr-\frac{r}{p}\|\gr u_{n}\|_{p}^{p}
=cr-\frac{r-p}{p}\|\gr u_{n}\|_{p}^{p}\leq cr.
\end{align*}
Let $n\ri\infty$ in the above inequality, we can infer that
\begin{align}\label{co3}
-\la\geq\frac{-cr}{a}>0.
\end{align}
One can derive from \eqref{co1} and \eqref{psca} that $\{u_{n}\}$ is a Palais-Smale sequence of $\Jl$,
where $\Jl(u)=\J(u)-\la\|u\|_{p}^{p}$. Then we can use Lemma \ref{spl} to conclude that there exist integer
$k\geq0$, a nonnegative function $u_{0}\in\X$, $k$ sequences $\{x_{n}^{i}\}\subset\R^{N}$ for $1\leq i\leq k$
such that $|x_{n}^{i}|\ri\infty$ as $n\ri\infty$ and 
\begin{align}\label{sp1}
u_{n}=u_{0}+\sum_{i=1}^{k}w_{\la}(x+y_{n}^{i})\quad\mbox{strongly in}\quad\Y,
\end{align}
$u_{0}$ is a solution of
\begin{align}
&\Lam u=\la|u|^{p-2}u+|u|^{r-2}u\quad\mbox{in}\quad\Omega,\label{sp2}\\
&a=\|u_{0}\|_{p}^{p}+k\|w_{\la}\|_{L^{p}(\R^{N})}^{p},\label{sp3}
\end{align}
and 
\begin{align}\label{sp4}
\Jl(u_{n})=\Jl(u_{0})+k\Il(w_{\la})+o_{n}(1).
\end{align}
By virtue of \eqref{ma2}, \eqref{sp3} and \eqref{sp4}, one has
\begin{align}\label{sp5}
c&=\J(u_{0})+km_{*}(\|w_{\la}\|^{p}_{L^{p}(\R^{N})})
\geq m_{*}(\|u_{0}\|_{p}^{p})+km_{*}(\|w_{\la}\|^{p}_{L^{p}(\R^{N})})\nonumber\\
&\geq\bigg(\frac{\|u_{0}\|^{p}_{L^{p}(\R^{N})}}{\|w\|_{L^{p}(\R^{N})}}\bigg)^{\frac{s_{p,r}+1}{s_{p,r}}}\I(w)
+k\bigg(\frac{\|w_{\la}\|^{p}_{L^{p}(\R^{N})}}
{\|w\|_{L^{p}(\R^{N})}}\bigg)^{\frac{s_{p,r}+1}{s_{p,r}}}\I(w)\\
&=\tau_{0}[(a-k\|w_{\la}\|^{p}_{L^{p}(\R^{N})})
^{\frac{s_{p,r}+1}{s_{p,r}}}+k(\|w_{\la}\|^{p}_{L^{p}(\R^{N})})^{\frac{s_{p,r}+1}{s_{p,r}}}],\nonumber
\end{align}
where $\tau_{0}=\frac{\I(w)}{(\|w\|_{L^{p}(\R^{N})}^{p})^{\frac{s_{p,r}+1}{s_{p,r}}}}$.

{\bf Step 1.} 
We claim that 
\begin{align}\label{ad}
\|w_{\la}\|^{p}_{L^{p}(\R^{N})}\geq\beta a \quad\mbox{if}\quad m_{*}(a)<c<2^{-1}m_{*}(a),
\end{align}
where
$$
\beta=p^{-1}\bigg(\frac{-r\I(w)}{\|w\|^{p}_{L^{p}(\R^{N})}}\bigg)^{s_{p,r}}\in(0,p^{-1}).
$$
\eqref{ma2} shows that
$$
 m_{*}(a)=(\frac{a}{\|w\|_{L^{p}(\R^{N})}})^{\frac{s_{p,r}+1}{s_{p,r}}}\I(w).
$$
By this fact, \eqref{ma1} and \eqref{co3}, we infer that
\begin{align}\label{cos11}
&\quad\|w_{\la}\|_{L^{p}(\R^{N})}^{p}=(-\la)^{s_{p,r}}\|w\|_{L^{p}(\R^{N})}^{p}
\geq (\frac{-cr}{a})^{s_{p,r}}\|w\|_{L^{p}(\R^{N})}^{p}\nonumber\\
&\geq \bigg(\frac{-rp^{-1/s_{p,r}}m_{*}(a)}{a}\bigg)^{s_{p,r}}\|w\|_{L^{p}(\R^{N})}^{p}
=p^{-1}\bigg(\frac{-r\I(w)}{\|w\|^{p}_{L^{p}(\R^{N})}}\bigg)^{s_{p,r}}a.
\end{align}
Since $w$ is a solution of \eqref{pi}, one concludes that
\begin{align*}
\|\gr w\|_{L^{p}(\R^{N})}^{p}+\|w\|_{L^{p}(\R^{N})}^{p}
=\|w\|_{L^{r}(\R^{N})}^{r}.
\end{align*}
Thereby we can deduce that
\begin{align*}
\I(w)=\frac{1}{p}\|\gr w\|_{\Lp}^{p}-\frac{1}{r}\|w\|_{L^{r}(\R^{N})}^{r}
=(\frac{1}{p}-\frac{1}{r})\|\gr w\|_{\Lp}^{p}-\frac{1}{r}\|w\|_{\Lp}^{p},
\end{align*}
which leads to that
\begin{align*}
\frac{-r\I(w)}{\|w\|_{\Lp}^{p}}=(1-\frac{r}{p})\frac{\|\gr w\|_{\Lp}^{p}}{\|w\|_{\Lp}^{p}}+1.
\end{align*}
Consequently, there holds 
\begin{align}\label{cos12}
0<\frac{-r\I(w)}{\|w\|_{\Lp}^{p}}<1.
\end{align}
In view of \eqref{cos11} and \eqref{cos12}, this claim is right.

{\bf Step 2.} We define 
$$
\beta_{1}=\beta^{\frac{s_{p,r}+1}{s_{p,r}}}+(1-\beta)^{\frac{s_{p,r}+1}{s_{p,r}}}
\in(2^{-\frac{1}{s_{p,r}}},1).
$$
We want to prove that if $m_{*}(a)<c<\beta_{1}m_{*}(a)$, then $k\leq1$.
For $k\geq1$, define the function 
$$
f_{k}(t):=(a-kt)^{\frac{s_{p,r}+1}{s_{p,r}}}+kt^{\frac{s_{p,r}+1}{s_{p,r}}}\quad\mbox{for any}\quad
t\in[0,\frac{a}{k}].
$$
We can compute that $f_{k}\in C^{1}([0,\frac{a}{k}])$ and
$$
f_{k}'(t)=\frac{s_{p,r}+1}{s_{p,r}}k[t^{\frac{1}{s_{p,r}}}-(a-kt)^{\frac{1}{s_{p,r}}}]\quad\mbox{for any}\quad
t\in[0,\frac{a}{k}].
$$ 
Take $t_{k}=\frac{a}{k+1}$. One deduces that $f_{k}$ is decreasing in $[0,t_{k}]$ and increasing in 
$[t_{k},\frac{a^{2}}{k}]$. Thus 
$$
\min_{t\in[0,\frac{a}{k}]}f_{k}(t)=f_{k}(t_{k})>0.
$$
From \eqref{sp3} and \eqref{ad}, we know $a\geq k\beta a$, which suggests that $k\leq \beta^{-1}$.
Define 
$$
g(k):=(1-k\beta)^{\frac{s_{p,r}+1}{s_{p,r}}}+k\beta^{\frac{s_{p,r}+1}{s_{p,r}}}\quad\mbox{for any}\quad k\in[1,\beta^{-1}].
$$
It is clear that $g\in C^{1}([1,\beta^{-1}])$ and 
$$
g'(k)=\beta^{\frac{s_{p,r}+1}{s_{p,r}}}-
\beta\frac{s_{p,r}+1}{s_{p,r}}(1-k\beta)^{\frac{1}{s_{p,r}}}\quad\mbox{for any}\quad k\in[1,\beta^{-1}].
$$
Let $k_{0}:=\frac{1}{\beta}-(\frac{s_{p,r}}{s_{p,r}+1})^{s_{p}}\in(1,\beta^{-1})$. We conclude that
$g$ is increasing in $[1,k_{0}]$ and  is decreasing in $[k_{0},\beta^{-1}]$.

If $k\geq2$, since $\|w_{\la}\|^{p}_{L^{p}(\R^{N})}\geq\beta a$ and $\beta\leq\frac{1}{2}$, by \eqref{ma2} and \eqref{sp5}, we conclude that
\begin{align*}
c&\geq\tau_{0}\max_{t\in[\beta a,\frac{a}{k}]}f_{k}(t)
=\tau_{0}\max\{f_{k}(\beta a),f_{k}(\frac{a}{k})\}
=\max\{g(k),k^{-1/s_{p,r}}\}m_{*}(a)\\
&\geq\max\{g(1),g(\beta^{-1}),2^{-1/s_{p,r}}\}m_{*}(a)
=\max\{\beta_{1},\beta^{\frac{1}{s_{p,r}}},2^{-1/s_{p,r}}\}m_{*}(a)=\beta_{1}m_{*}(a).
\end{align*}
It is impossible due to  $c<2^{-1}m_{*}(a)$. Hence we have proved that $k\leq1$.

{\bf Step 3.} We shall prove that there exists $\xi=\xi(a)\in[\beta_{1},1)$ such that if $m_{*}(a)<c<\xi m_{*}(a)$, then $k=0$.

From Step $2$ we know that $k\leq1$. Suppose by contradiction that $u_{0}=0$. Then it follows from 
\eqref{sp3} and \eqref{sp5} that $c=0$ or $c=m_{*}(a)$, which is a contradiction. Thus $u_{0}\not\equiv0$.
\eqref{sp2} leads to 
\begin{align}\label{sps31}
\|\gr u_{0}\|_{p}^{p}+(-\la)\|u_{0}\|_{p}^{p}=\|u_{0}\|_{r}^{r}.
\end{align}
Applying \eqref{sps31} and Sobolev inequality, we know that there exists $C>0$ independent of $u_{0}$
and $\la$ such that
\begin{align*}
\min\{1,-\la\}(\|\gr u_{0}\|_{p}^{p}+\|u_{0}\|_{p}^{p})
\leq\|u_{0}\|_{r}^{r}\leq C(\|\gr u_{0}\|_{p}^{p}+\|u_{0}\|_{p}^{p})^{\frac{r}{p}}.
\end{align*} 
In view of $u_{0}\not\equiv0,\,\la<0$ and $r>p$, we can infer that
\begin{align}\label{sps32}
\|\gr u_{0}\|_{p}^{p}+\|u_{0}\|_{p}^{p}
\geq \bigg(\frac{\min\{1,-\la\}}{C}\bigg)^{\frac{p}{r-p}}>0.
\end{align}
Using \eqref{sps31}, \eqref{sps32} and \eqref{gn}, we derive that
$$
\|\gr u_{0}\|_{p}^{p}\leq\|u_{0}\|_{r}^{r}\leq
C_{N,p,r}^{p}\|u_{0}\|_{p}^{(1-\gamma_{p,r})r}\|\gr u_{0}\|_{p}^{r\gamma_{p,r}}.
$$
It holds from $p<r<p_{c}$ that $0<r\gamma_{p,r}<p$. Then 
\begin{align}\label{sps33}
\|\gr u_{0}\|_{p}\leq C_{N,p,r}^{\frac{r}{p-\gamma_{p,r}}}\|u_{0}\|_{p}^{\frac{(1-\gamma_{p,r})r}{p-r\gamma_{p,r}}}.
\end{align}

In the light of \eqref{sps32}, \eqref{sps33}, \eqref{co3} and $c<2^{-1/s_{p,r}}m_{*}(a)$, one can deduce that
\begin{align}\label{sps34}
C_{N,p,r}^{\frac{pr}{p-r\gamma_{p,r}}}(\|u_{0}\|_{p}^{p})^{\frac{(1-\gamma_{p,r})p}{p-r\gamma_{p,r}}}
+\|u_{0}\|_{p}^{p}
\geq\bigg(\frac{1}{C}\min\{1,-r2^{-1/s_{p,r}}\}\frac{m_{*}(a)}{a}\bigg)^{\frac{p}{r-p}}>0.
\end{align}
\end{proof}
We can see from \eqref{sps34} that there exists $C(a)>0$ such that
$$
\|u_{0}\|_{p}^{p}\geq C(a).
$$
Define 
\begin{align*}
\xi(a):=
\begin{cases}
\beta_{1}&\quad\mbox{if}\quad a\leq C(a),\\
\max\{\beta_{1},\frac{f_{1}(a-C(a))}{a^{\frac{s_{p,r}+1}{s_{p,r}}}}\}
&\quad\mbox{if}\quad a\geq C(a).
\end{cases}
\end{align*}

If $a\leq C(a)$, we can conclude that $\|u_{0}\|_{p}^{p}\geq C(a)\geq a$. According to 
\eqref{sp3}, we get $\|u_{0}\|_{p}^{p}=a$ and $k=0$.

If $a>C(a)$, we have that $0<a-C(a)<a$ and
$$
0<f_{1}(a-C(a))<\max\{f_{1}(0),f_{1}(a)\}=a^{\frac{s_{p,r}+1}{s_{p,r}}},
$$
which implies that 
$$
0<\frac{f_{1}(a-C(a))}{a^{\frac{s_{p,r}+1}{s_{p,r}}}}<1.
$$
Thus we derive $\xi(a)\in[\beta_{1},1)$. Assume by contradiction that $k=1$. By virtue of \eqref{sp3}, \eqref{ad}  and $\|u_{0}\|_{p}^{p}\geq C(a)$, we conclude that
$$
\beta a\leq\|w_{\la}\|_{L^{p}(\R^{N})}^{p}=a-\|u_{0}\|_{p}^{p}\leq a-C(a).
$$
From this fact, \eqref{ma2}, \eqref{sp5} and $\xi(a)\in[\beta_{1},1)$, one deduces that
\begin{align*}
c\geq\tau_{0}\max\{f_{1}(\beta a),f_{1}(a-C(a))\}\geq \xi(a)m_{*}(a).
\end{align*}
It is a contradiction due to $c<\xi(a)m_{*}(a)$. Then $k=0$ and $u_{n}\ri u_{0}$ in $\X$ by using Lemma \ref{spl}.

\begin{prop}\label{comm}
Assume that $\{u_{n}\}\subset S(a)$ is a minimizing sequence of $\J$ on $S(a)$. Then there exists a sequence
$\{y_{n}\}\subset\R^{N}$ such that $u_{n}(x)=w_{\la}(x-y_{n})+o_{n}(1)$ in $\Y$
and $|y_{n}|\ri\infty$ as $n\ri\infty$.
\end{prop}
\begin{proof}
Observe that $\I(u_{n})=\J(u_{n})$ and $m(a)=m_{*}(a)$. Then $\{u_{n}\}$ is a minimizing sequence of $\I$ on $S_{*}(a)$. It follows from the Ekeland variational principle that there exists a Palas-Smale sequence $\{v_{n}\}$
for $\I|_{S(a)}$ such that
$$
\I(v_{n})\ri m_{*}(a)\quad\mbox{and}\quad u_{n}=v_{n}+o_{n}(1).
$$ 
By Proposition \ref{ma}, there exists $\{y_{n}\}\subset\R^{N}$ such that
$$
v_{n}(\cdot+y_{n})\ri w_{\la}\quad\mbox{in}\quad\Y.
$$
Since $u_{n}=v_{n}+o_{n}(1)$, then 
\begin{align}\label{comm1}
u_{n}\ri w_{\la}(\cdot-y_{n})\quad\mbox{in}\quad\Y.
\end{align}
Suppose by contradiction that $\{y_{n}\}$ is bounded. Then up to a subsequence we have that
$y_{n}\ri y_{0}$ for some $y_{0}\in\R^{N}$. Note that
\begin{align*}
w_{\la}(\cdot-y_{n})\wri w_{\la}(\cdot-y_{0})\quad\mbox{in}\quad\Y
\end{align*}
and
\begin{align*}
\|w_{\la}(\cdot-y_{n})\|\ri\|w_{\la}(\cdot-y_{0})\|.
\end{align*}
Thus 
\begin{align}\label{comm2}
w_{\la}(\cdot-y_{n})\ri w_{\la}(\cdot-y_{0})\quad\mbox{in}\quad\Y.
\end{align}
One can derive from  $u_{n}=v_{n}+o_{n}(1)$, \eqref{comm1} and \eqref{comm2} that 
$$
u_{n}\ri w_{\la}(\cdot-y_{0}) \quad\mbox{in}\quad\Y.
$$
By the fact that $u_{n}\equiv0$ in $\R^{N}\setminus\Omega$, we conclude that $w_{\la}=0$ in $\R^{N}\setminus\Omega$. It is impossible since $w_{\la}>0$ in $\R^{N}$. Thus $|y_{n}|\ri\infty$.
\end{proof}

\section{Existence of solution}\label{sec4}

In this section we are devoted to show the existence of normalized solution for \eqref{p}.

\begin{lem}\label{up}
For any $a>0$, there exists $R_{a}>0$ such that for any $0<R\leq R_{a}$, it holds
$$
\sup_{y\in\R^{N}}\I(\Psi_{a,R}(y))<\xi m_{*}(a),
$$
 where $\xi=\xi(a)\in(0,1)$ is give by Lemma \ref{co}.
\end{lem}
\begin{proof}
In view of Proposition \ref{ma}, we have $\I(w_{\la})=m_{*}(a)<0$.
Let $\xi_{0}\in(\xi,1)$. Noting that $\I$ is continuous, then there exists $\delta>0$ such that for any
$u\in \Y$ and $\|u-w_{\la}\|_{\Y}<\delta$, one has
$$
\I(u)<\xi_{0}m_{*}(a).
$$
Exploiting Lemma \ref{be}, we can get that there exists $R_{a}>0$ such that for any $0<R<R_{a}$,
$$
\|\Psi_{a,R}(y)-w_{\la}(\cdot-y)\|_{\Y}<\delta\quad\mbox{for any}\quad y\in\R^{N}.
$$
It leads to that
$$
\|\Psi_{a,R}(y)(\cdot+y)-w_{\la}\|_{\Y}<\delta\quad\mbox{for any}\quad y\in\R^{N}.
$$
Consequently, there holds that
$$
\I(\Psi_{a,R}(y))=\I(\Psi_{a,R}(y)(\cdot+y))<\xi_{0} m_{*}(a)\quad\mbox{for any}\quad y\in\R^{N}.
$$
The proof is finished.
\end{proof}

In what follows we construct a min-max  theorem to derive a Palais-Smale sequence.

To prove Theorem \ref{thm2}, we suppose that $\R^{N}\backslash\Omega\subset B_{R_{a}}(0)$, where $R_{a}$ is defined as Lemma \ref{up}. Then $\Psi_{a,R_{a}}\in S(a)$.
Let $\tau>0$ be such that $\R^{N}\setminus\Omega\subset B_{r}(0)$ and $\chi\in C(\R^{+},\R)$ be a nonincreasing function defined as
\begin{align*}
\chi(t)
\begin{cases}
1,&\quad 0<t\leq \tau,\\
\frac{\tau}{t},&\quad t>\tau.
\end{cases}
\end{align*}
We introduce the barycenter function $\mu:\Y\ri\R^{N}$ defined as
$$
\mu(u):=\int_{\R^{N}}|u(x)|^{p}\chi(|x|)xdx.
$$
It is clear that
$$
\chi(|x|)|x|\leq r\quad\mbox{for any}\quad x\in\R^{N}.
$$

Define the set
$$
\mathcal{B}:=\{u\in S(a):\mu(u)=0\}.
$$
Since $\Psi_{a,R_{a}}(0)\in S(a)$ and $w_{\la}$ is a radial function, then $\Psi_{a,R_{a}}(0)\in\mathcal{B}$
and $\mathcal{B}$ in not empty. Let
$$
b(a):=\inf_{u\in\mathcal{B}}\J(u).
$$

The following estimate result holds.
\begin{lem}\label{esb}
There holds that
$$
b(a)>m_{*}(a).
$$
\end{lem}
\begin{proof}
From the definition of $m_{*}(a)$, $b(a)$ and Lemma \ref{min}, we have 
$$
m(a)=m_{*}(a)\leq b(a).
$$
Suppose by contradiction that $b(a)=m_{*}(a)$. Then there exists $\{u_{n}\} \subset S(a)$ such that
\begin{align*}
\mu(u_{n})=0\quad\mbox{for any}\quad n\geq1\quad\mbox{and}
\quad \J(u_{n})\ri m_{*}(a).
\end{align*}
Without loss of generality, we may assume that $u_{n}\geq0$ almost for each $x\in\Omega$ due to
$\mu(|u_{n}|)=\mu(u_{n})$ and $\J(|u_{n}|)=\J(u_{n})$. Applying Proposition \ref{comm}, there exists $y_{n}\in\R^{N}$ such that
$$
u_{n}=w_{\la}(\cdot-y_{n})+o_{n}(1)\quad\mbox{strongly in}\quad\Y\quad
\mbox{and}\quad|y_{n}|\ri\infty.
$$ 
In order to obtain a contradiction, we define the sets
$$
(\R^{N})_{n}^{+}:=\{x\in\R^{N}:\langle x,y_{n}\rangle>0\}
\quad\mbox{and}\quad
(\R^{N})_{n}^{-}:=\R^{N}\backslash(\R^{N})_{n}^{+}.
$$ 
By the continuity of $w_{\la}$ and $w_{\la}(0)=\max_{x\in\R^{N}}w_{\la}(x)$, we can show that there exists $\tau_{0}>0$ such that
\begin{align}\label{esb1}
w_{\la}(x-y_{n})\geq\frac{1}{2}w_{\la}(0)>0\quad\mbox{for any}\quad x\in B_{\tau_{0}}(y_{n}).
\end{align}
Since $|y_{n}|\ri\infty$, there exists $n_{0}>0$ such that for any $x\in B_{\tau_{0}}(y_{n})$,
\begin{align}\label{esb3}
\langle x,y_{n}\rangle=\frac{|x|^{2}+|y_{n}|^{2}-\langle x-y_{n},x-y_{n}\rangle}{2}
>\frac{|y_{n}|^{2}-r_{0}}{2}\geq\frac{|y_{n}|^{2}}{4}>0\quad\mbox{for any}\quad n\geq n_{0}.
\end{align}
According to \eqref{esb1}, \eqref{esb3} and the meaning of $\chi$, there holds for $n\geq n_{0}$,
\begin{align}\label{esb4}
&\quad\int_{(\R^{N})^{+}_{n}}|w_{\la}(x-y_{n})|^{p}\chi(|x|)\langle x,y_{n}\rangle dx\nonumber\\
&=\int_{B_{\tau_{0}}(y_{n})}|w_{\la}(x-y_{n})|^{p}\chi(|x|)\langle x,y_{n}\rangle dx
+\int_{(\R^{N})_{n}^{+}\setminus B_{\tau_{0}}(y_{n})}|w_{\la}(x-y_{n})|^{p}\chi(|x|)\langle x,y_{n}\rangle dx\nonumber\\
&\geq\frac{w_{\la}^{p}(0)}{2^{p}}\frac{\tau_{0}}{|y_{n}|}\frac{|y_{n}|^{2}}{4}|B_{\tau_{0}}(0)|
\geq \frac{\tau_{0}w_{\la}^{p}(0)|y_{n}|}{2^{p+2}}|B_{\tau_{0}}(0)|.
\end{align}

On the other hand, we conclude that for any $x\in(\R^{N})_{+}^{-}$,
\begin{align*}
|x-y_{n}|^{2}=|x|^{2}+|y_{n}|^{2}-2\langle x,y_{n}\rangle\geq |x|^{2}+|y_{n}|^{2}.
\end{align*}
Thus one concludes that $|x-y_{n}|\geq|x|$. By the fact that $|x-y_{n}|\geq|x|$, the monotonicity of $w_{\la}$
and $\chi(|x|)x\leq \tau$ for any $x\in\R^{N}$, we infer that
$$
|w_{\la}(x-y_{n})|^{p}\chi(|x|)|x|\leq \tau|w_{\la}(x)|^{p}\in L^{1}(\R^{N}).
$$
Besides, note that $w_{\la}(x-y_{n})\ri 0$ for each $x\in\R^{N}$ as $n\ri\infty$. Then we can use the Lebesgue 
dominated convergence theorem to deduce that
\begin{align}\label{esb5}
\int_{(\R^{N})_{n}^{-}}|w_{\la}(x-y_{n})|\chi(|x|)|x|dx=o_{n}(1).
\end{align} 

In the light of \eqref{esb4} and \eqref{esb5}, one has that
\begin{align}\label{esb6}
&\quad\langle\mu(w_{\la}(x-y_{n})),\frac{y_{n}}{|y_{n}|}\rangle\nonumber\\
&=\int_{(\R^{N})_{n}^{+}}|w_{\la}(x-y_{n})|^{p}\chi(|x|)\frac{\langle x,y_{n}\rangle}{|y_{n}|}dx
+\int_{(\R^{N})_{n}^{-}}|w_{\la}(x-y_{n})|^{p}\chi(|x|)\frac{\langle x,y_{n}\rangle}{|y_{n}|}dx\nonumber\\
&\geq \frac{\tau_{0}w_{\la}^{p}(0)}{2^{p+3}}|B_{\tau_{0}}(0)|
-\int_{(\R^{N})_{n}^{-}}|w_{\la}(x-y_{n})|^{p}\chi(|x|)|x|dx\nonumber\\
&=\frac{\tau_{0}w_{\la}^{p}(0)}{2^{p+3}}|B_{\tau_{0}}(0)|+o_{n}(1).
\end{align}

But $\mu$ is continuous, $\mu(u_{n})=0$ and $u_{n}-w_{\la}(\cdot-y_{n})\ri0$ in $\Y$, we have 
$\mu(w_{\la}(\cdot-y_{n}))\ri0$ as $n\ri\infty$. Thereby we obtain a contradiction due to \eqref{esb6}.
\end{proof}

\begin{lem}\label{ta}
There exists $R_{a}'>R_{a}$ such that
$$
\J(\Psi_{a,R_{a}}(y))<b(a)\quad\mbox{and}\quad
\langle \mu(\Psi_{a,R_{a}}(y)),y\rangle>0\quad\mbox{for}\quad y\in\R^{N}\,\,\mbox{with}\,\,|y|\geq R_{a}'.
$$
\end{lem}
\begin{proof}
Recall that $\R^{N}\backslash\Omega\subset B_{R_{a}}(0)$. Then there holds that $\Psi_{a,R_{a}}(y)\in S(a)$ 
and $\J(\Psi_{a,R}(y))=\I(\Psi_{a,R}(y))$ for any $y\in\R^{N}$. By virtue of Lemma \ref{be}-$(ii)$ and Lemma \ref{esb}, we infer that there exists $R_{a}'>R_{a}$ such that
$$
\J(\Psi_{a,R_{a}}(y))<b_{a}\quad\mbox{for any}\quad y\in\R^{N}\quad\mbox{and}\quad|y|\geq R_{a}'.
$$

Proceeding as \eqref{esb6}, for $|y|\geq R_{a}'$ large enough we conclude that
\begin{align*}
&\quad\langle\mu(\Psi_{a,R}(y)),\frac{y}{|y|})\rangle_{\R^{N}}\nonumber\\
&=C_{a,\la}\int_{(\R^{N})_{n}^{+}}|\eta(\frac{x}{R_{a}})w_{\la}(x-y)|^{p}\chi(|x|)\frac{\langle x,y\rangle}{|y|}dx
+C_{a,\la}\int_{(\R^{N})_{n}^{-}}|\eta(\frac{x}{R_{a}})w_{\la}(x-y)|^{p}\chi(|x|)\frac{\langle x,y\rangle}{|y|}dx\nonumber\\
&\geq \frac{\tau_{0}w_{\la}^{p}(0)}{2^{p+3}}|B_{\tau_{0}}(0)|
-C_{a,\la}\int_{(\R^{N})_{n}^{-}}|\eta(\frac{x}{R_{a}})w_{\la}(x-y_{n})|^{p}\chi(|x|)|x|dx\nonumber\\
&=\frac{\tau_{0}w_{\la}^{p}(0)}{2^{p+3}}|B_{\tau_{0}}(0)|+o_{n}(1),
\end{align*}
where 
$$
C_{a,\la}:=\frac{\sqrt[p]{a}}{\|\eta(\cdot/R_{a})w_{\la}(\cdot-y)\|_{L^{p}(\R^{N})}}.
$$
So we can deduce that
$$
\langle \mu(\Psi_{a,R_{a}}(y)),y\rangle>0\quad\mbox{for}\quad y\in\R^{N}\,\,\mbox{with}\,\,|y|\geq R_{a}'.
$$
\end{proof}

Let us introduce the minimax class
$$
\Gamma:=\{\gamma\in C(\overline{B_{R_{a}'}(0)},S(a)):\gamma(y)=\Psi_{a,R_{a}}(y)\quad\mbox{for any}
\quad y\in\partial B_{R_{a}'}(0)\}.
$$
Define the associated energy level
$$
d:=\inf_{\gamma\in\Gamma}\sup_{y\in B_{R_{a}'}(0)}\J(\gamma(y)).
$$

There holds that $\Gamma$ in not empty and $d$ is well-defined since $\Psi_{a,R_{a}}\in\Gamma$.

\begin{lem}\label{bu}
One has that $b(a)\leq d<\xi m_{*}(a)$.
\end{lem}
\begin{proof}
It follows from $\Psi_{a,R_{a}}\in S(a)$, the definition of $d$ and Lemma \ref{up} that
$$
d\leq \sup_{y\in B_{R_{a}'}}\J(\Psi_{a,R_{a}}(y))<\xi m_{*}(a).
$$
Next we show that $b(a)\leq d$. For any $\gamma\in\Gamma$, we define the map
$H:[0,1]\times\overline{B_{R_{a}'}(0)}\ri\R^{N}$ by
$$
H(t,y)=(1-t)\mu(\gamma(y))+ty.
$$
Then $H$ is continuous.
One can derive from Lemma \ref{ta} that for any $(t,y)\in[0,1]\times\partial B_{R_{a}'}(0)$,
$$
\langle H(t,y),y\rangle=(1-t)\langle\mu(\gamma(y)),y\rangle
+t\langle y,y\rangle>0,
$$
which leads to that $0\not\in H([0,1],\partial B_{R_{a}'}(0))$. Then by the homotopy invariance of Brouwer degree, we can infer that
\begin{align*}
\deg (\mu\circ\gamma,B_{R_{a}'}(0),0)=\deg(\text{id},B_{R_{a}'}(0),0)=1,
\end{align*}
where $\text{id}$ is the identity map. Thus there exists $\hat{y}\in B_{R_{a}'}(0)$ such that 
$\mu(\gamma(\hat{y}))=0$. Then $\gamma(\hat{y})\in\mathcal{B}$. So we can conclude that for any
$\gamma\in\Gamma$,
$$
b(a)\leq\J(\gamma(\hat{y}))\leq\sup_{y\in B_{R_{a}'}(0)}\J(\gamma(y)).
$$
By the arbitrariness of $\gamma$, there holds $b(a)\leq d$. 
\end{proof}

\begin{proof}[{\bf Proof of Theorem \ref{thm2}}]
Let $\Sigma:=\{\Psi_{a,R_{a}}(y):y\in \partial B_{R_{a}'}(0)\}\subset S(a)$ be a compact set. In view of Lemma \ref{esb}, Lemma \ref{ta} and Lemma \ref{bu}, one has that
\begin{align}\label{thm21}
m_{*}(a)<\sup_{u\in \Sigma}\J(u)<b(a)\leq d<\xi m_{*}(a).
\end{align}
We consider the class 
$$
\mathcal{F}:=\{\gamma(\overline{B_{R_{a}'}(0)}):\gamma\in\Gamma\}.
$$
Then $\mathcal{F}$ is a homotopy-stable family of compact subsets of $S(a)$ with closed boundary $\Sigma$, see
\cite[Definition 3.1]{G1993}.
In fact, take $\gamma\in\Gamma$ and $h\in C([0,1]\times S(a),S(a))$ satisfying
$$
h(t,u)=u\quad\mbox{for any}\quad (t,u)\in\{0\}\ast S(a)\cup [0,1]\times \Sigma.
$$
Let $\hat{h}:B_{R_{a}}'\ri S(a)$ defined as $\hat{h}=h(1,\gamma(\cdot))$. By the definition of $\Gamma$ and the meaning of $h$, we have that for any $y\in \partial B_{R_{a}'}(0)$,
$$
\hat{h}(y)=h(1,\gamma(y))=h(1,\Psi_{a,R_{a}}(y))=\Psi_{a,R_{a}}(y),
$$
which implies $\hat{h}(B_{R_{a}}')\in\mathcal{F}$. Thus we proved that $\mathcal{F}$ is a homotopy-stable family of compact subsets of $S(a)$ with closed boundary $\Sigma$.

Applying \cite[Theorem 3.2]{G1993}, there exists sequence $\{u_{n}\}\subset S(a)$ such that
$\J(u_{n})\ri d$ and uniformly for $\varphi\in\X$ with $\|\varphi\|_{\X}\leq1$,
$$
\langle\J'(u_{n}),\varphi\rangle=\la_{n}\int_{\Omega}|u_{n}|^{p-2}u_{v}\varphi dx+o_{n}(1).
$$
It follows from Lemma \ref{co} and Lemma \ref{bu} that there exists $(u,\la)\in\X\times(0,\infty)$ such that
$$
(u_{n},\la_{n})\ri (u,\la)\quad\mbox{in}\quad\X\times\R.
$$
\end{proof}

\begin{proof}[{\bf Proof of Theorem \ref{thm3}}]
Let $R:=\max_{x\in\R^{N}\setminus\Omega}|x|+1$. Then $\R^{N}\setminus\Omega\subset B_{R}(0)$. We denote that for $\theta\in(0,1)$, 
$$
\Omega_{\theta}:=\{\theta x:x\in\Omega\}.
$$
Thus we can conclude that $\R^{N}\setminus\Omega\subset B_{\theta R}(0)$ for any $\theta\in(0,1)$.
Let $R_{a}$ be given in Theorem \ref{thm2} with $a=1$ and $a_{*}=(\frac{R_{1}}{R})^{s_{p,r}}$.
Set $\theta:=a^{1/ps_{p,r}}$. We have $\theta R\leq R_{1}$, which leads to 
$\R^{N}\backslash\Omega_{\theta}\subset B_{R_{1}}(0)$. It follows from Theorem \ref{thm2} that there exists
$(u_{1},\la_{1})\in W_{0}^{1,p}(\Omega_{\theta})\times(-\infty,0)$ such that
$u_{1}>0$ in $\Omega_{\theta}$, $\la_{1}<0$ and $(u_{1},\la_{1})$ satisfies 
\begin{align*}
\begin{cases}
\Lam u=\la |u|^{p-2}u+|u|^{r-2}u\quad\mbox{in}\quad\Omega_{\theta},\\
\int_{\Omega_{\theta}}|u|^{p}dx=1.
\end{cases}
\end{align*}
Define $u(x):=\theta^{\frac{p}{r-p}}u_{1}(\theta x)$ for any $x\in\Omega$. Then we can conclude taht
$(u,\la)$ is a solution of \eqref{p} with $\la=\theta^{p}\la_{1}<0$. 
\end{proof}

\section{Multiplicity of solutions}\label{sec5}

Let $\Omega=\R^{N}\setminus\overline{B_{1}(0)}$ and let $\Xr$ and $\Yr$ denote the set of all radial functions
in $\X$ and $\Y$ respectively. Take
$$
m_{\rad}:=\inf_{u\in S_{\rad}(a)}\J(u),
$$
where 
$$
S_{\rad}(a):=\{u\in\Xr:\|u\|_{p}^{p}=a\}.
$$

For any $u\in\Xr$, we have that $\mu(u)=0$, which implies that $u\in\mathcal{B}$ for any $u\in\Xr$. Then one can deduce from Lemma \ref{esb} that
$$
-\infty<m_{*}(a)<b(a)\leq m_{\rad}(a)<\infty.
$$

Recall the notation of genus. Let 
$$
\mathcal{A}:=\{A\subset S_{\rad}(a): A\,\mbox{is closed in}\, S_{\rad}(a),\,A=-A\}.
$$
For any nonempty set $A\in\mathcal{A}$, the genus $\mathcal{G}(A)$ of $A$ is defined by the least integer $k\geq1$ for which there exists an odd continuous mapping $\kappa:A\ri\R^{k}\setminus\{0\}$. Let
$\mathcal{G}(A)=+\infty$ if such an integer does not exist and $\mathcal{G}(A)=0$ if $A=\emptyset$.
We define for any $k\in\mathbb{N}^{+}$ that 
$$
\mathcal{A}_{k}:=\{A\in\mathcal{A}:\mathcal{G}(A)\geq k\}.
$$
If $\mathcal{A}_{k}\neq\emptyset$, we define
$$
c_{k}:=\inf_{A\in\mathcal{A}_{k}}\sup_{u\in A}\J(u).
$$

Next we introduce an abstract minimax theorem given by \cite[Theorem 2.1]{JL2019}.

\begin{lem}\label{Je}(\cite[Theorem 2.1]{JL2019})
Let $\Phi:\Xr\ri\R$ be an even functional of class $C^{1}$. Assume that $\J|_{S_{\rad}(a)}$ is bounded from below
and satisfies the $(PS)_{c}$ condition for all $c<0$, and that $\mathcal{A}_{k}\neq\emptyset$ for each $k\in\mathbb{N}^{+}$. Then a sequence of minimax values $-\infty<c_{1}\leq c_{2}\leq\cdot\cdot\cdot\leq c_{k}\leq\cdot\cdot\cdot$ can be defined as follows:
$$
c_{k}:=\inf_{A\in\mathcal{A}_{k}}\sup_{u\in A}\Phi(u)\quad\mbox{for}\quad k\geq1,
$$ 
and the following statements hold.

$(i)$ $c_{k}$ is a critical value of $\Phi|_{S_{\rad}(a)}$ provided $c_{k}<0$;

$(ii)$ Denote by $K_{c}$ the set of critical points of $\Phi|_{S_{\rad}(a)}$ at a level $c\in\R$. If
$$
c_{k}=c_{k+1}=\cdot\cdot\cdot=c_{k+l-1}:=c<0\quad\mbox{for some}\quad k,\,l\geq1,
$$ 
then $\mathcal{G}(K_{c})\geq l$. In particular, $\Phi|_{S_{\rad}(a)}$ has infinitely many critical points at the level $c$ if $l\geq2$;

$(iii)$ If $c_{k}<0$ for all $k\geq1$, then $c_{k}\ri0^{-}$ as $k\ri\infty$.
\end{lem}

Take $\Phi=\J$. We can easily see that $\J$ is an even functional of class $C^{1}$ and $\J|_{S_{\rad}(a)}$ is bounded from below.

\begin{lem}\label{cor}
$\J|_{S_{\rad}(a)}$ satisfies the $(PS)_{c}$ condition for $c\in\R$.
\end{lem} 
\begin{proof}
Assume that $\{u_{n}\}\subset S_{\rad}(a)$ fulfills that $\J(u_{n})\ri c$ and
$(\J|_{S_{\rad}(a)})'(u_{n})\ri0$. We can conclude that $\{u_{n}\}$ is bounded in $\Xr$. Thus in the sense of a subsequence there exists $u\in\Xr$ such that $u_{n}\wri u$ in $\Xr$ and $u_{n}\ri u$ in $L^{t}(\Omega)$
for any $t\in(p,p^{*})$. Similar to the proof of Lemma \ref{co}, there exists $\la<0$ such that for every
$\varphi\in \Xr$,
\begin{align}\label{cor1}
\int_{\Omega}|\gr u_{n}|^{p-2}\gr u_{n}\gr\varphi dx
-\la\int_{\Omega}|u_{n}|^{p-2}u_{n}\varphi dx-\int_{\Omega}|u_{n}|^{r-2}u_{n}\varphi dx
=o_{n}(1)\|\varphi\|_{\X}.
\end{align}
Define the map
$$
V(x,y):=\langle |x|^{p-2}x-|y|^{p-2}y,x-y\rangle\quad\mbox{for any}\quad x,\,y\in\R^{N}.
$$
Let $\varphi_{R}:\R^{N}\ri\R$ be a smooth function and radially symmetric such that $\varphi_{R}=1$ in $B_{R}(0)$,
$\varphi_{R}=0$ in $\R^{N}\setminus B_{2R}(0)$, $0\leq\varphi_{R}\leq1$ and $|\gr\varphi_{R}(x)|\leq\frac{C}{R}$.
By \eqref{cor1} and the fact that $\varphi_{R}(u_{n}-u)\wri0$ in $\Xr$, there holds that 
$\langle\Jl'(u_{n})-\Jl'(u),\varphi_{R}(u_{n}-u)\rangle=o_{n}(1)$, which leads to
\begin{align}\label{cor2}
&\quad\int_{\R^{N}}(V(\gr u_{n},\gr u)-\la V(u_{n},u))\varphi_{R}dx
+\int_{\R^{N}}(|\gr u_{n}|^{p-2}\gr u_{n}-|\gr u|^{p-2}\gr u)(u_{n}-u)\gr \varphi_{R}dx\nonumber\\
&=\int_{\R^{N}}(|u_{n}|^{r-2}u_{n}-|u|^{r-2}u)(u_{n}-u)\varphi_{R}dx+o_{n}(1).
\end{align}
Since $|\gr\varphi_{R}(x)|\leq\frac{C}{R}$, $\{u_{n}\}$ is bounded in $\Xr$ and $u_{n}\ri u$ in $L_{\text{loc}}^{p}(\R^{N})$, by the H\"older inequality one infers that
\begin{align*}
&\quad\bigg|\int_{\R^{N}}(|\gr u_{n}|^{p-2}\gr u_{n}-|\gr u|^{p-2}\gr u)(u_{n}-u)\gr \varphi_{R}dx\bigg|\\
&\leq\frac{C}{R}\int_{\R^{N}}(|\gr u_{n}|^{p-1}+|\gr u|^{p-1})|u_{n}-u| dx\\
&\leq\frac{C}{R}\bigg(\int_{\R^{N}}|u_{n}|^{p}dx+\int_{\R^{N}}|u|^{p}dx\bigg)^{\frac{p-1}{p}}
\bigg(\int_{B_{2R}(0)}|u_{n}-u|^{p}dx\bigg)^{\frac{1}{p}}\ri0\quad\mbox{as}\quad n\ri\infty.
\end{align*}
Similarly we can prove that
$$
\int_{\R^{N}}(|u_{n}|^{r-2}u_{n}-|u|^{r-2}u)(u_{n}-u)\varphi_{R}dx\ri0\quad\mbox{as}\quad n\ri\infty.
$$
Thus we conclude from \eqref{cor2} that
\begin{align}\label{cor3}
\int_{B_{R}(0)}V(\gr u_{n},\gr u)dx\leq o_{n}(1).
\end{align}
It follows from \cite[Lemma 2.1]{D1998} that for any $x,\,y\in\R^{N}$,
\begin{align}\label{cor4}
\langle |x|^{p-2}x-|y|^{p-2}y,x-y\rangle\geq C_{0}|x-y|^{p}
\quad\mbox{for any}\quad p\geq2.
\end{align}
Thus  we can deduce from \eqref{cor3} and \eqref{cor4} that
$$
\int_{B_{R}(0)}|\gr u_{n}-\gr u|^{p}dx=o_{n}(1).
$$
By the arbitrariness of $R$ we conclude that
$$
\gr u_{n}\ri\gr u\quad\mbox{a.e. in}\quad\R^{N}.
$$
From this fact we have that
\begin{align*}
\int_{\Omega}|\gr u|^{p} dx
-\la\int_{\Omega}|u|^{p}dx=\int_{\Omega}|u|^{r} dx+o_{n}(1).
\end{align*}
Note that 
\begin{align*}
\int_{\Omega}|\gr u_{n}|^{p} dx
-\la\int_{\Omega}|u_{n}|^{p}dx=\int_{\Omega}|u_{n}|^{r} dx+o_{n}(1).
\end{align*}
By Br\'ezis-Lieb lemma, there holds that
\begin{align*}
\int_{\Omega}|\gr u_{n}-\gr u|^{p} dx
-\la\int_{\Omega}|u_{n}-u|^{p}dx=\int_{\Omega}|u_{n}-u|^{r} dx+o_{n}(1)=o_{n}(1).
\end{align*}
Then $u_{n}\ri u$ in $\Xr$.
\end{proof}

Next we  want to prove that $\mathcal{A}_{k}\neq\emptyset$ and  $c_{k}<0$ for any integer $k>0$. 
\begin{lem}\label{eq}
For any integer $k\geq1$, there exists $\{u_{1},\cdot\cdot\cdot,u_{k}\}\subset S_{\rad}(a)$
such that $\int_{\Omega}u_{i}u_{j}dx=0$ for $i\neq j$ and
$$
T_{k}\subset S_{\rad}(a),\quad\max_{u\in T_{k}}\J(u)<0,
$$
where 
$$
T_{k}:=\bigg\{t_{1}u_{1}+\cdot\cdot\cdot+t_{k}u_{k}: t_{i}\in\R,\,i=1,\cdot\cdot\cdot,k,\,\sum_{i=1}^{k}|t_{i}|^{p}=1\bigg\}.
$$
Then $\mathcal{A}_{k}\neq\emptyset$, $c_{k}\in(-\infty,0)$ and $m_{\rad}(a)\in(-\infty,0)$.
\end{lem}
\begin{proof}
For each $k\geq1$ let $\{v_{1},\cdot\cdot\cdot,v_{k}\}\subset C_{c}^{\infty}(\R^{N}\setminus\overline{B_{2}(0)})\cap\Yr$ be a nonnegative sequence such that
\begin{align*}
\int_{\R^{N}}|v_{i}|^{p}dx=a,\,i=1,\,2,\cdot\cdot\cdot,k,\quad
\supp v_{i}\cap\supp v_{j}=\emptyset\quad\mbox{for}\quad i\neq j,\,i,\,j=1,\cdot\cdot\cdot,k.
\end{align*} 
Let $\eta:\R^{N}\ri\R$ be a function satisfying that
$$
\eta\equiv0\quad\mbox{for}\quad |x|\leq1,\quad\eta(x)\equiv1\quad\mbox{for}\quad|x|\geq2,
$$
and let $\sigma:\R\times\Xr\ri\Xr$ be defined as
$$
\sigma(t,v)=\eta(x)e^{\frac{N}{p}t}v(e^{t}x)\quad\mbox{for any}\quad (t,x)\in\R\times\Xr.
$$
It is easy to deduce that 
$$
\supp \sigma(t,v_{i})\cap\sigma(t, v_{j})=\emptyset\quad\mbox{for any}\quad t\in\R,\quad
i\neq j,\quad i,\,j=1,\cdot\cdot\cdot,k.
$$
Using $\supp v_{i}\subset \R^{N}\setminus \overline{B_{2}(0)}$ and by the change of variable $y=e^{t}x$, we have 
$$
\int_{\R^{N}}|\sigma(t,v_{i})|^{p}dx
=\int_{\R^{N}}e^{Nt}|v_{i}(e^{t}x)|^{p}dx
\int_{\R^{N}}e^{Nt}|v_{i}(y)|^{p}e^{-Nt}dy=a\quad\mbox{for}\quad i=1,\cdot\cdot\cdot,k.
$$
Thus $\sigma(t,v_{i})\subset S_{\rad}(a)$. We also can show that
$$
\int_{\R^{N}}|\gr \sigma(t,v_{i})|^{p}dx=e^{pt}\int_{\R^{N}}|v_{i}|^{p}dx
$$
and 
$$
\int_{\R^{N}}|\sigma(t,v_{i})|^{r}dx=e^{\frac{(r-p)N}{p}t}\int_{\R^{N}}|v_{i}|^{r}dx.
$$
Then for any $t_{i}\in\R$ with $\sum_{i=1}^{k}|t_{i}|^{p}=1$, one can conclude that
$$
\int_{\Omega}|\sum_{i=1}^{k}t_{i}\sigma(t,v_{i})|^{p}dx
=\sum_{i=1}^{k}t_{i}^{p}\int_{\Omega}|\sigma(t,v_{i})|^{p}dx=a
$$
and
\begin{align*}
\J\bigg(\sum_{i=1}^{k}t_{i}\sigma(t,v_{i})\bigg)&=
\sum_{i=1}^{k}\bigg(\frac{|t_{i}|^{p}}{p}e^{pt}\int_{\R^{N}}|\gr v_{i}|^{p}dx-
\frac{|t_{i}|^{r}}{r}e^{\frac{(r-p)N}{p}t}\int_{\R^{N}}|v_{i}|^{r}dx\bigg)\\
&\leq\frac{e^{pt}}{p}\bigg(\max_{1\leq i\leq k}\|\gr v_{i}\|_{p}^{p}-
e^{(\frac{(r-p)N}{p}-p)t}C\min_{1\leq i\leq k}\|v_{i}\|_{r}^{r}\bigg),
\end{align*}
where $C=\min_{|t_{1}|^{p}+\cdot\cdot\cdot+|t_{k}|^{p}=1}>0$.
Due to $\frac{(r-p)N}{p}-p<0$, then
$$
\J\bigg(\sum_{i=1}^{k}t_{i}\sigma(t,v_{i})\bigg)<0\quad\mbox{for}\quad t<0\quad\mbox{and}
\quad |t|\quad\mbox{large enough}.
$$

Take $u_{i}=\sigma(t,v_{i})$ with $t<0$ such that |t| large enough and
$$
T_{k}:=\bigg\{t_{1}u_{1}+\cdot\cdot\cdot+t_{k}u_{k}: t_{i}\in\R,\,i=1,\cdot\cdot\cdot,k,\,\sum_{i=1}^{k}|t_{i}|^{p}=1\bigg\}.
$$ 
We have that
$$
T_{k}\subset S_{\rad}(a),\quad\max_{u\in T_{k}}\J(u)<0.
$$
Since $T_{k}\in\mathcal{A}$ is homeomorphic to $\mathbb{S}^{k-1}:=\{z\in\R^{k}:|z|=1\}$, one can derive from the property of genus that $\mathcal{G}(T_{k})=\mathcal{G}(\mathbb{S}^{k-1})=k$, which suggests that
$T_{k}\subset \mathcal{A}_{k}$ and $\mathcal{A}_{k}\neq\emptyset$. Besides, we have
$$
m_{\rad}(a)\leq c_{k}\leq\max_{u\in T_{k}}\J(u)<0.
$$
From Lemma \ref{esb}, there holds 
$$
-\infty<m_{*}(a)<m_{\rad}(a).
$$
\end{proof}

\begin{proof}[{\bf Proof of Theorem \ref{thm4}}]
From Lemma \ref{Je}, Lemma \ref{cor} and Lemma \ref{eq}, the proof can be finished.
\end{proof}

\end{document}